\documentclass[a4paper,12pt,english]{amsart}
\usepackage[latin1]{inputenc}

\usepackage{calc} \usepackage{graphicx,palatino,verbatim}

\usepackage{amsmath}

\usepackage[T1]{fontenc}

\usepackage[english]{babel}

\usepackage{newlfont}

\usepackage{amsthm}

\usepackage{amssymb}

\usepackage{amsmath}

\usepackage{eufrak}

\usepackage{latexsym}

\usepackage{young} \usepackage[mathscr]{eucal}

\usepackage{amsfonts}    
\usepackage{syntonly}   
\usepackage{graphicx}
\usepackage{mathptmx}       
\usepackage{subfig}       
\usepackage{url}

\usepackage[letterpaper]{geometry}

\input xy \xyoption{all}

\newtheorem{theo}{Theorem}[section]

\newtheorem{remark}[theo]{Remark}

\newtheorem{lemma}[theo]{Lemma}
\newtheorem{definition}[theo]{Definition}

\newtheorem{proposition}[theo]{Proposition}

           \DeclareMathOperator{\Sym}{Sym}
\DeclareMathOperator{\Alt}{Alt}         \DeclareMathOperator{\Aut}{Aut}
\DeclareMathOperator{\PSL}{PSL}           \DeclareMathOperator{\GL}{GL}
\DeclareMathOperator{\AGL}{AGL}

\renewcommand{\phi}[0]{\varphi}

\begin{document}
  \title[Translation based  ciphers over arbitrary  finite fields]{On
    the   group  generated   by 
    the  round   functions  of
    translation based  ciphers over arbitrary  finite fields}
\author{R.~Aragona}          

\address[Aragona]{Department  of  Mathematics,  University of  Trento,  Via
  Sommarive 14, 38123 Trento, Italy}
\email{ric.aragona@gmail.com}

\author{A.~Caranti}
    \address[Caranti]{Department  of
  Mathematics,  University of  Trento,  Via 
  Sommarive 14, 38123 Trento, Italy}
\email{andrea.caranti@unitn.it} 
 \author{F.~Dalla     Volta}
 \address[Dalla Volta]{Department of
  Mathematics  and  Applications,  University of  Milano-Bicocca,  Via
  R. Cozzi, 53, 20126 Milano, Italy}
   \email{francesca.dallavolta@unimib.it}

\author{M.~Sala} \address[Sala]{Department  of  Mathematics,  University of  Trento,  Via
  Sommarive 14, 38123 Trento, Italy}
\email{maxsalacodes@gmail.com} 

\date{22 July 2013, Version 3.14}

\thanks{The second and third authors are members of INdAM-GNSAGA, Italy.}
\thanks{The  first, second and  fourth authors  thankfully acknowledge
  support  by the  Department of  Mathematics, University  of Trento.}
\thanks{The  third  author  thankfully  acknowledges  support  by  the
  Department   of   Mathematics   and  Applications,   University   of
  Milano-Bicocca   and   MIUR-Italy  via   PRIN   `Group  theory   and
  applications''.}

\begin{abstract} 
We define a  translation based cipher over an  arbitrary finite field,
and study  the permutation group  generated by the round  functions of
such a  cipher. We show  that under certain  cryptographic assumptions
this  group  is primitive.  Moreover,  a  minor  strengthening of  our
assumptions allows us  to prove that such a group  is the symmetric or
the alternating  group; this improves  upon a previous result  for the
case of characteristic two.
\end{abstract}


\maketitle

\section{Introduction}
Translation based ciphers, as defined by Caranti, Dalla Volta and Sala
in~\cite{CGC-cry-art-carantisalaImp}, form  a class of  iterated block
ciphers, i.e.\  obtained by  the composition of  several key-dependent
permutations    of   the    message/cipher   space    called   ``round
functions''. This  class of  ciphers contains well-known  ciphers like
AES~\cite{CGC-cry-art-deamenrijmen1}                                and
SERPENT~\cite{CGC-cry-art-serpent}.

In                 1975                 Coppersmith                and
Grossman~\cite{CGC-cry-art-copgro1975generators}    investigated   the
permutation group generated by the round functions of a cipher, aiming
at finding  properties of  the group which  imply weaknesses  of the
cipher.        In       this        direction,        Kaliski       et
al.~\cite{CGC-cry-art-kalburrivshe1988data} proved  that if this group
is too small,  then the cipher is vulnerable  to certain cryptanalytic
attacks.  Paterson~\cite{CGC-cry-art-paterson1}  showed  that if  this
group is imprimitive,  then it is possible to embed  a trapdoor in the
cipher.

In~\cite{CGC-cry-art-carantisalaImp} the authors provided  conditions   on the
S-Boxes  of a  translation based  cipher  which ensure  that the  group
generated   by   its   round   functions   is   primitive.    Moreover
in~\cite{CGC-cry-art-cardalsal2009application},  using the O'Nan-Scott
classification  of primitive  groups, it  was  proved that  if such  a
cipher satisfies  some additional cryptographic  assumptions, then the
group is the alternating or the symmetric  group.

In       this       paper       we      extend       the       results
of~\cite{CGC-cry-art-carantisalaImp}
and~\cite{CGC-cry-art-cardalsal2009application}  to  translation based
ciphers defined over an arbitrary  finite field. The main point is the
move from  vector spaces defined over the  field $\mathbb{F}_{2}$ with
two elements, to vector  spaces defined over a field $\mathbb{F}_{p}$,
where $p$  is an arbitrary prime.  The extension to  vector spaces $V$
over an arbitrary finite field  $\mathbb{F}_{q}$, where $q$ is a power
of  the prime  $p$, is  then  quite straightforward,  but requires  to
consider  $V$   also  as   a  vector  space   over  the   prime  field
$\mathbb{F}_{p}$. The latter structure is completely determined by
the   structure  of   $(V,   +,  0)$   as   an  (elementary)   abelian
group. Therefore,  when we  will be speaking  of a  \emph{subspace} of
$V$, we will mean an  $\mathbb{F}_{q}$-subspace, while we will use the
term  \emph{subgroup} to  refer to  an  $\mathbb{F}_{p}$-subspace. The
related     terminology     is      explained     in     detail     in
Section~\ref{sec:preliminaries}. Readers wishing  to take a quick look
at  the paper are  advised to  think of  the particular  case $q  = p$
throughout.

Block ciphers using algebraic structures other than the field with two
elements  have already  been studied.  For example,  Biham  and Shamir
in~\cite{CGC-cry-book-bihsha1993differential} studied  the security of
DES against a  differential attack when some of  the operations in DES
are        replaced       by       addition        modulo       $2^n$.
In~\cite{CGC-cry-art-patramsun2003luby}  Patel,  Ramzan  and  Sundaram
showed  that Luby-Rackof  Ciphers are  secure against  adaptive chosen
plaintext   and  ciphertext  attacks,   and  have   better  time/space
complexity if considered over the  prime field $\mathbb{F}_p$ for $p >
2$.   Some    result   in   this   direction    are   also   contained
in~\cite{CGC-cry-art-babbomcolmorsco2012}.

This  paper  is  organised  as  follows. After  some  preliminaries  in
Section~\ref{sec:preliminaries},   we   introduce  translation   based
ciphers  over  arbitrary finite  fields  in Section~\ref{sec:tbc}.  In
Section~\ref{sec:prim}      we      prove      first     a      result
(Theorem~\ref{primitivity})  about the  primitivity of  the  group $G$
generated by the round functions of a translation based cipher over an
arbitrary  finite field.  In our  main result  (Theorem~\ref{main}) we
then  show that  this group  $G$ is  actually the  alternating  or the
symmetric group. To prove this,  we follow the scheme arising from the
O'Nan-Scott  classification of  primitive  groups, in  a special  case
dealt   with  by   Li   (Theorem~\ref{Li}),  showing   that  all   the
possibilities except  the alternating or symmetric group  can be ruled
out (Sections~\ref{sec:as}--\ref{sec:wp}).

\section{Preliminaries}      
\label{sec:preliminaries}
Let $G$  be a finite group acting  transitively on a set  $V$. We will
write the action of $g\in G$ on an element $v\in V$ as $v g$.

A   \textit{partition}   $\mathcal{B}$   of   $V$  is   said   to   be
$G$-\textit{invariant}     if     $B g\in\mathcal{B}$,    for     every
$B\in\mathcal{B}$   and  $g\in  G$.   A  partition   $\mathcal{B}$  is
\textit{trivial} if  $\mathcal{B}=\{V\}$ or $\mathcal{B}=\{\{v\}\, |\,
v\in V\}$. A non-trivial  $G$-invariant partition $\mathcal{B}$ of $V$
is called a  \textit{block system} for the action of  $G$ on $V$. Each
$B\in\mathcal{B}$  is  called a  \textit{block  of imprimitivity}.  (A
block of  imprimitivity can be characterised  as a subset  $B$ of $V$,
which is not a singleton  or the whole of $V$, such that  for all $g \in
G$ either $B = B g$, or $B \cap B g = \emptyset$.) We will say that $G$
is  \textit{imprimitive} in its  action on  $V$ if  it admits  a block
system. A useful elementary fact is
\begin{lemma}\label{lemma:block}
A block  of imprimitivity is of  the form $v  H$, for some $v  \in V$,
where  $H$  is  a  proper  subgroup of  $G$  properly  containing  the
stabiliser of $v$  in $G$.\end{lemma} If $G$ is  not imprimitive, then
$G$ is called \textit{primitive}.

In the rest  of the paper $p$ is  a prime, $q=p^f$ is a  power of $p$,
and  $V$ is  a vector  space of  dimension $d$  over the  finite field
$\mathbb{F}_q$.

We will be  regarding $V$ also as a vector space  over the prime field
$\mathbb{F}_p$ of dimension $e = d f$. Since the latter structure is completely determined by
the (elementary) abelian group structure $(V, +, 0)$, we will refer to
$\mathbb{F}_p$-subspaces  of  $V$  as  \emph{subgroups}, and  we  will
reserve  the  term  \emph{subspace}  for  $\mathbb{F}_q$-subspaces  of
$V$.  Similarly,  a  function  on   $V$  is  \emph{linear}  if  it  is
$\mathbb{F}_q$-linear,  and \emph{additive}  if  it is
$\mathbb{F}_p$-linear. 

We 
denote by  $\GL(V)$ the  group of linear  permutations of $V$,  and by
$\GL(V,  +,  0)$  (or  simply  $\GL(V,  +)$)  the  group  of  additive
permutations of  $V$. We denote by $\AGL(e,p)$ the  group of affine
transformations  of a vector space of dimension $e$ over
$\mathbb{F}_{p}$. This is used in Section~\ref{sec:prim} and in
Section~\ref{sec:affine}, where some more related notation is introduced.

We write  $\Sym(V)$ and $\Alt(V)$  respectively for the  symmetric and
the  alternating  group on  the  set  $V$. For  $v  \in  V$, we  write
$\sigma_{v} \in \Sym(V)$ for the translation $x \mapsto x + v$ on $V$,
and  denote   by  $T(V)=\{\sigma_v  \mid   v\in  V\}$  the   group  of
translations  on  $V$. Clearly  $T(V)$  is  a  transitive subgroup  of
$\Sym(V)$.  Because  of   Lemma~\ref{lemma:block},  $T(V)$  is  always
imprimitive unless $f=1$ (that is, $q = p$ is prime) and $d=1$.
\begin{lemma}
If $f > 1$, or $d > 1$,  a block system for $T(V)$ is of
  the form $\{W+v \mid v\in V\}$, for some proper, non-zero
  \emph{subgroup} $W$ of $V$.
\end{lemma} 

We recall  that a  group $G$ of  permutations acting  on a set  $V$ is
called \textit{regular}  if, given $v\in  V$, for each $w\in  V$ there
exists  a \emph{unique}  $g\in  G$ such  that  $v g=w$ (in  particular,
regularity implies transitivity). The  group $T(V)$ of translations is
regular.


\section{Translation based block ciphers over $\mathbb{F}_{q}$}
\label{sec:tbc}

We  consider block  ciphers  defined over  an  arbitrary finite  field
$\mathbb{F}_{q}$.
 
Let $\mathcal{C}$ be a block cipher for which that the plaintext space
$V=(\mathbb{F}_{q})^d$, for some  $d\in\mathbb{N}$, coincides with the
ciphertext space. According to Shannon~\cite{Sha49},
\begin{equation*}\label{eq:Sha}
\mathcal{C} = \{ \tau_k \mid k \in \mathcal{K} \}
\end{equation*}
is a  set of permutations $\tau_k$  of $V$; here  $\mathcal{K}$ is the
key space.

It   would    be   very    interesting   to   determine    the   group
$\Gamma(\mathcal{C}) = \langle \tau_k  \mid k \in \mathcal{K} \rangle \le
\Sym(V)$ generated  by the permutations  $\tau_k$.  Unfortunately, for
many  classical   cases  (e.g.\   AES~\cite{CGC-cry-art-deamenrijmen1},
SERPENT~\cite{CGC-cry-art-serpent}, DES~\cite{CGC-misc-nistDESfips46},
IDEA~\cite{CGC-cry-art-laimassey2006IDEA})  this
appears to be a difficult problem. However, more manageable overgroups
of  $\Gamma$  have been  investigated  (see~\cite{We93, HoWe94,  WE02,
  SW08}), such as the ones that we now define.

Suppose  that  each  element   of  the  cipher  $\mathcal{C}$  is  the
composition  of  $l$  round   functions,  that  is, permutations
$\tau_{k,1},\dots, \tau _{k,l}  $, where each  $\tau_{k,h}$
is determined by a master key $k \in \mathcal{K}$, and the round index
$h$. Define the groups
\begin{equation*}
\Gamma_h(\mathcal{C})   =  \langle  \tau_{k,h}   \mid  k\in\mathcal{K}
\rangle \le \Sym(V),
\end{equation*}
 for each $h$, and the group
\begin{equation*}
\Gamma_\infty(\mathcal{C})  = \langle \Gamma_h(\mathcal{C})  \mid h  = 1,
\dots ,  l\rangle =  \langle \tau_{k,h} \mid  k\in\mathcal{K}, h  = 1,
\dots, l \rangle
\end{equation*}
In  the  literature, ``round''  often  refers  either  to the  ``round
index'' or to the ``round function''.


Consider a direct sum decomposition of $V$
\begin{equation}\label{eq:directsum}
V=V_1\oplus\cdots\oplus V_n
\end{equation}
where  $n  >   1$, and the $V_i$ are subspaces of $V$ with
$\dim_{\mathbb{F}_{q}}(V_i)=m>1$,  for  each 
$i\in\{1,\ldots,n\}$,  so that $d = m n$.  Each $v\in  V$ can then be written
uniquely as $v=v_1 + \cdots +  v_n$, with $v_i\in V_i$.

\begin{definition}
An   element   $\gamma\in\Sym(V)$   is   called   a   \emph{bricklayer
  transformation}  with  respect  to~\eqref{eq:directsum} if  $\gamma$
acts on  an element $v=v_1 + \cdots +  v_n$,  with $v_i\in V_i$,
as
$$ v\gamma=v_1\gamma_1 + \cdots +  v_n\gamma_n,
$$  for  some  $\gamma_i\in\Sym(V_i)$.  Each $\gamma_i$  is  called  a
\emph{brick}.
\end{definition}
\begin{definition}
A linear map $\lambda\in\GL(V)$ is called a \emph{proper mixing layer}
if   no   sum   of   the   $V_i$,   except   $\{0\}$   and   $V$,   is
$\lambda$-invariant.
\end{definition}
Now  we   generalise  the  definition  of   translation  based  cipher
$\mathcal{C}$  (Definition  3.1  in~\cite{CGC-cry-art-carantisalaImp},
when $\mathcal{C}$ is a block cipher over $\mathbb{F}_{q}$.

\begin{definition}\label{deftb}
  A block cipher  $\mathcal{C} = \{ \tau_k \mid k  \in \mathcal{K} \}$ over
  $\mathbb{F}_{q}$ is called \emph{translation based (tb)} if
  \begin{enumerate}
  \item[(1)]  each  $\tau_k$  is  the  composition  of  $l$  round  functions
    $\tau_{k,h}$, for $k \in \mathcal{K}$,  and $h = 1, \dots, l$, where
    in  turn  each round  function  $\tau_{k,h}$  can  be written  as  a
    composition   $\gamma_h\lambda_h\sigma_{\phi(k,    h)}$   of   three
    permutations of $V$, where
    \begin{itemize}
    \item $\gamma_h$  is a bricklayer transformation not  depending on $k$
      and $0\gamma_h=0$,
    \item $\lambda_h$ is a linear permutation not depending on $k$,
    \item $\phi : \mathcal{K}  \times \{1, \dots , l \} \to  V$ is the key
      scheduling function, so that  $\phi(k, h)$ is the $h$-th \emph{round
        key}, given  the master key  $k$; 
    \end{itemize}
  \item[(2)] for at least one round index $h_0$ we have that
    \begin{itemize}
    \item $\lambda_{h_0}$ is a proper mixing layer, and
    \item the map $\mathcal{K} \to V$ given by $k \mapsto \phi(k, h_0)$ is
      surjective,  that is,  every element  of $V$  occurs as  an $h_0$-th
      round key.
    \end{itemize}
  \end{enumerate}
  We  will refer  to a  round that  satisfies condition~(2) as a
  \emph{proper round}.
\end{definition}

We now  work in  the group $\Gamma_h(\mathcal{C})$,  for a  fixed $h$,
omitting   for   simplicity   the   indices  $h$   for   the   various
functions. Write $\rho=\gamma\lambda$.
Note the following
\begin{lemma}\label{lemma:transide}
Suppose that for  a certain $h$, the map $\mathcal{K}  \to V$ given by
$k \mapsto \phi(k, h)$ is surjective. Then
\begin{equation*}
\Gamma_h(\mathcal{C}) = \langle \rho, T(V) \rangle.
\end{equation*} 
\end{lemma}

\begin{proof}
By assumption, $\Gamma_h(\mathcal{C}) =  \langle \rho\sigma_r : r\in V
\rangle$. Thus  $\rho =  \rho \sigma_0 \in  \Gamma_h(\mathcal{C})$, so
that all $\sigma_{v} \in  \Gamma_h(\mathcal{C})$, and the statement is
clear.
\end{proof}

\begin{lemma}\label{rem1}
Suppose that for  a certain $h$, the map $\mathcal{K}  \to V$ given by
$k \mapsto \phi(k, h)$ is surjective.

Then if $\Gamma_h(\mathcal{C})$ is  imprimitive on $V$, a block system
$\mathcal{B}$ is of  the form $\{W+v \mid v\in  V\}$, for some proper,
non-trivial subgroup $W$ of $V$.
\end{lemma}

\begin{proof}
By  Lemma~\ref{lemma:transide},  $\Gamma_h(\mathcal{C})$ contains  the
group   $T(V)$  of   translations.  If   $\Gamma_h(\mathcal{C})$  acts
imprimitively on $V$, so  does $T(V)$, By Lemma~\ref{lemma:block}, the
block containing $v \in  V$ is of the form $W +  v$, for $W$ a proper,
non-trivial subgroup of $T(V)$.
\end{proof}


\begin{proposition}\label{imp}
Suppose that for  a certain $h$, the map $\mathcal{K}  \to V$ given by
$k \mapsto \phi(k, h)$ is surjective.

Then $\Gamma_h(\mathcal{C})$ is imprimitive if and only if there exists
a proper, non-trivial  subgroup $W$ of $V$ such that  for all $v\in V$
and $u\in W$, we have
$$ (u+v)\gamma-v\gamma\in W\lambda^{-1}.
$$
\end{proposition}
\begin{proof}
  $\Gamma_h(\mathcal{C})$  is imprimitive if  and only  if there  is a
  block  system of  type  $\{W+v \mid  v\in  V \}$,  for some  proper,
  non-trivial subgroup $W$. Thus we have
  \begin{equation*}
    (W+v)\rho=W+v\gamma\lambda\sigma_0\Longrightarrow
    (W+v)\gamma\lambda=W+v\gamma\lambda,
  \end{equation*}
  for every $v\in V$. Hence, for all $u\in W$ and $v\in V$ we have
  $$ (u+v)\gamma\lambda-v\gamma\lambda \in W,
  $$ so that
  $$ (u+v)\gamma-v\gamma\in W\lambda^{-1}.
  $$
\end{proof}

\section{Primitivity}
\label{sec:prim}

We   generalise  the   definition  of   weak  uniformity   and  strong
anti-invariance,    given     for    vectorial    Boolean    functions
in~\cite{CGC-cry-art-carantisalaImp},       to       any      function
$f\,:\,A\rightarrow A$, where  $A$ is a vector space  of dimension $m$
over a prime  field $\mathbb{F}_{p}$, that is an elementary abelian group of order $p^m$.

Let  $a\in  A$.  For   every  $f\,:\,A\rightarrow  A$,  we  denote  by
$\hat{f}_a$ the function
$$
\begin{array}{rcl}
\hat{f}_a:A&\longrightarrow&A\\ x&\longmapsto&f(x+a)-f(x).
\end{array}
$$ Let $\mathrm{Im}(\hat{f}_a)=\{y\in  A \mid y=\hat{f}_a(x)\mbox{ for
  some }x\in A\}$ be the image of $\hat{f}_a$.

\begin{definition}
For $m\geq  2$ and  $\delta\geq p$, let  $A$ be  a subgroup of  $V$ of
order  $p^m$, and  $f\in \Sym(A)$.  We  say that  $f$ is  \emph{weakly
  $\delta$-uniform} if for every $a\in A\setminus\{0\}$ we have
$$ |\mathrm{Im}(\hat{f}_a)|>\frac{p^{m-1}}{\delta}.
$$
\end{definition}
Recently, weakly 2-uniform functions for  4 bits have been studied and
classified in~\cite{CGC2-cry-art-Font12}.
\begin{remark}\label{rem2}
If a  function $f$ is  weakly $\delta$-uniform, with  $\delta\leq p^r$
for some  $r\in\mathbb{N}$, and $\mathrm{Im}(\hat{f}_a)$  is contained
in a subgroup $W$ of $A$, then $\lvert W \rvert \ge p^{m-r}$.
\end{remark}
\begin{definition}
Let  $A$  be  a  subgroup  of  $V$. We  say  that  $f\in  \Sym(A)$  is
\emph{strongly $r$-anti-invariant}  if for  any two subgroups  $U$ and
$W$  such that  $f(U)=W$, we  have  either $\lvert  U \rvert=\lvert  W
\rvert < p^{m-r}$ or $U=W=A$.
\end{definition}
We prove the main result of this section
\begin{theo}\label{primitivity}
Let $\mathcal{C}$ be a tb cipher over $\mathbb{F}_{q}$.

Suppose that the $h$-th round is proper, and let $1\leq r<\frac{m}{2}$.

If each brick of $\gamma_h$ is
\begin{itemize}
\item[(1)] weakly $p^r$-uniform, and
\item[(2)] strongly $r$-anti-invariant,
\end{itemize}
then    $\Gamma_h(\mathcal{C})$    is    primitive   and    so    also
$\Gamma_\infty(\mathcal{C})$ is primitive.
\end{theo}

\begin{proof}
For the sake of simplicity we drop the round indices.

We suppose,  by way of contradiction,  that $\Gamma_h(\mathcal{C})$ is
imprimitive. By Lemma~\ref{rem1}, the  blocks of imprimitivity are the
cosets of a subgroup of $V$. Let $U$ be a proper, non-trivial subgroup
of $V$ such  that $\{ U+v \mid v\in V \}$ is a block  system for $G$. Since
$\gamma\lambda\sigma_0=\gamma\lambda\in   \Gamma_h(\mathcal{C})$,   we
have $U\gamma\lambda=U+v$, for some $v\in V$. But $0\gamma\lambda=0\in
U+v$, so
\begin{equation}\label{eq1}
U\gamma\lambda=U.
\end{equation}
Let $\pi_i:V\rightarrow V_i$ such that $\pi_i(v)=v_i$ and let $I$ be the
set of all $i$ such that $\pi_i(U)\neq 0$.

If $U\cap V_i=V_i$ for every $i\in I$, then $U=\bigoplus_{i\in I} V_i$
and so, by definition of $\gamma$, $U\gamma=U$. Hence, by (\ref{eq1}),
$U\lambda=U$, contradicting the assumption that $\lambda$ is a
proper mixing layer.

Therefore we can  suppose that there exists $i\in  I$ such that $U\cap
V_i\neq    V_i$.    We    write    $W=U\gamma=U\lambda^{-1}$.    Since
$\Gamma_h(\mathcal{C})$ is imprimitive, by Proposition~\ref{imp} we have
\begin{equation}\label{eq3}
\hat{\gamma}_u(v)\in W
\end{equation}
for every $u\in U$ and $v\in V$.

Moreover, we note that
\begin{equation}\label{eq2}
(U\cap V_i)\gamma_i=W\cap V_i.
\end{equation}
We  denote $\gamma_i$  with  $\gamma'$. By  (\ref{eq3})  we have  that
$\mathrm{Im}(\hat{\gamma}'_u)\subseteq  W\cap V_i$  for all  $u\in U\cap
V_i$.  By  hypothesis  $\gamma'$  is  weakly $p^r$-  uniform,  so,  by
Remark~\ref{rem2}, $\lvert W\cap V_i  \rvert = \lvert U\cap V_i \rvert
\geq  p^{m-r}$.  But,  by   (\ref{eq2}), this contradicts the
assumption that
$\gamma'$ is strongly $r$-anti-invariant.
\end{proof}

We are now able to state our main result.
\begin{theo}\label{main}
Let $d = m n$, with $m,n>1$. Let $\mathcal{C}$ be a tb cipher such that
\begin{itemize}
\item[(1)]     $\mathcal{C}$    satisfies     the     hypothesis    of
  Theorem~\ref{primitivity}, and
\item[(2)]        for        all        $0\ne        a\in        V_i$,
  $\{(x+a)\gamma_i-x\gamma_i\;:\;x\in  V_i\}$  is  not  a coset  of  a
  subgroup of $V_i$.
\end{itemize}
Then the  group $G = \Gamma_{\infty}(\mathcal{C})$  is either $\Alt(V)$
or $\Sym(V)$.
\end{theo}

In~\cite{CGC-cry-art-carantisalaImp} it is shown that the hypothesis of Theorem~\ref{main} are satisfied by well-known ciphers like AES and SERPENT.

We  know  from  Theorem~\ref{primitivity}  that the  subgroup  $G$  of
$\Sym(V)$ is  primitive.  We  are thus able  to apply  the O'Nan-Scott
classification     of      primitive     groups.      Actually,     by
Lemma~\ref{lemma:transide},   $G$  contains   the   group  $T(V)$   of
translations, which acts regularly on $V$.
We  are  then   able  to  appeal  to  a   result  of  Li~\cite[Theorem
  1.1]{CGC-alg-art-li03}  for primitive  groups containing  an abelian
regular subgroup. In  the particular case when the degree  of $G$ is a
power of a prime, this states the following.
\begin{theo}[\cite{CGC-alg-art-li03}, Theorem 1.1]\label{Li}
Let $G$ be a primitive group of degree $p^b$, with $b> 1$. Suppose $G$
contains a regular abelian subgroup $T$.

Then $G$ is one of the following.
\begin{itemize}
\item[(1)] Affine, $G \leq  \AGL(e,p)$, for some prime $p$ and
  $e \ge 1$.
\item[(2)] Wreath product, that is
$$ G \cong (S_1\times\cdots\times S_t).O.P,
$$ with  $p^b=c^t$ for  some $c$ and  $t>1$. Here  $T=T_1\times \cdots
  \times  T_t$,  with  $T_i\leq  S_i$  and  $|T_i|=c$  for  each  $i$,
  $S_1\cong\ldots\cong                   S_t$,                  $O\leq
  \mathrm{Out}(S_1)\times\cdots\times\mathrm{Out}(S_t)$,  $P$ permutes
  transitively the $S_i$, and one of the following holds:
\begin{itemize}
\item[(i)]            $(S_i,T_i)=(\PSL_2(11),\mathbb{Z}_{11})$,
  $(S_i,T_i)=(M_{11},\mathbb{Z}_{11})$,
  $(S_i,T_i)=(M_{23},\mathbb{Z}_{23})$;
\item[(ii)] $S_i=\Sym(c)$  or $\Alt(c)$, and  $T_i$ is
  an abelian group of order $c$.
\end{itemize}
\item[(3)]  Almost  simple, that  is,  $S\leq  G\leq  \Aut(S)$, for  a
  nonabelian simple group $S$.
\end{itemize}
\end{theo}

Here the notation  $S.T$ denotes an extension of the  group $S$ by the
group $T$.

In the next three Sections we will examine the three cases of Theorem~\ref{Li}, and show that the only possibilities for our $G$ is
 to be the full alternating  or symmetric group.  (Note that these two
 groups  fall  under  the   almost  simple  case.)   This  will  prove
 Theorem~\ref{main}.

\section{The almost simple case}
\label{sec:as}
In  the almost  simple case  (3),  we note  that $S$  is a  transitive
subgroup  of  the  primitive  group  $G$, so  the  intersection  of  a
one-point stabiliser  in $G$ with $S$  is a proper subgroup  of $S$ of
index $p^{b}$, where $b=f m n$ with  $m,n>1$, i.e. $b>3$. By Theorem 1 and
Section  (3.3) in~\cite{CGC-alg-art-guralnick83}, the  only nonabelian
simple  groups that have  a subgroup  of index  $p^b$ with  $b>3$, are    the    alternating   group    and   the    group
$\PSL_\alpha(\beta)$, where
\begin{itemize}
\item[(i)] $(\beta^\alpha-1)/(\beta-1)=p^{b}$,
\item[(ii)] $\alpha$ is a prime, and
\item[(iii)] $\beta$ is a power  of a prime $\pi$ such that $\pi\equiv
  1\mod \alpha$.
\end{itemize}
If  $S=\Alt(p^{b})$,  since  $\Aut(\Alt(p^{b}))=\Sym(p^{b})$,  then  $G$  is
either $\Alt(p^{b})$ or $\Sym(p^{b})$.

In our case, we can  rule out $S=\PSL_\alpha(\beta)$ as follows. First
we note that by (iii),  we have $\beta^{i}\equiv 1\mod \alpha$, for each
$i$. So
$$
(\beta^\alpha-1)/(\beta-1)=\beta^{\alpha-1}+\beta^{\alpha-2}+\cdots+\beta+1\equiv
\alpha\mod \alpha,
$$                i.e.,                $\alpha$                divides
$\beta^{\alpha-1}+\beta^{\alpha-2}+\cdots+\beta+1$,  and then,  by (i)
we have that $\alpha$  divides $p^b$. Hence $\alpha=p$, since $\alpha$
is  a   prime.   By  (iii)   we  have  $\beta=k p + 1$  for   some  $k\in
2\mathbb{N}$, therefore
$$
\beta^{p-1}+\beta^{p-2}+\cdots+\beta+1=(k p + 1)^{p-1}+(k p +
1)^{p-2}+\cdots+(k p + 1)+1=p^{b}.
$$ So we have
\begin{equation}\label{p^d}
\left(\sum_{j=1}^{p-1}\sum_{i=1}^{p-j}\binom{p-j}{i}k^i p^{i} \right)+p=p^{b}.
\end{equation}
Dividing (\ref{p^d}) by $p$, we obtain
\begin{equation}\label{p^d-1}
\left(\sum_{j=1}^{p-1}\sum_{i=1}^{p-j}\binom{p-j}{i}k^i p^{i-1}\right)+1=p^{b-1}.
\end{equation}
Rewrite (\ref{p^d-1}) as follows
\begin{equation}\label{p^d-1-1}
\left(\sum_{j=1}^{p-1}\sum_{i=2}^{p-j}\binom{p-j}{i}k^i p^{i-1}\right)+\left(\sum_{j=1}^{p-1}\binom{p-j}{1}k\right)=p^{b-1}-1.
\end{equation}
Since
$$
\sum_{j=1}^{p-1}\binom{p-j}{1}k=\sum_{j=1}^{p-1}(p-j)k=\sum_{j=1}^{p-1}
j k=\frac{(p-1)p k}{2},
$$ then the  left side of (\ref{p^d-1-1}) is divisible  by $p$ and the
right side of (\ref{p^d-1-1}) is  not divisible by $p$. So we conclude
that $(\beta^{p}-1)/(\beta-1)=p^{b}$,  is not possible if  $b>1$, which is
our hypothesis.


\section{The affine case}
\label{sec:affine}

As observed  by Li~\cite{CGC-alg-art-li03},  if $V =  (V, +, 0)$  is a
vector  space over  the  field $\mathbb{F}_p$  ,  the symmetric  group
$\Sym(V)$ will contain in general many isomorphic copies of the affine
group. The obvious one, $\AGL(V,  +, 0)$, consists of the maps $x
\mapsto x g + v$, where $g \in \GL(V, +, 0)$, and $v \in V$. But  there
are  in  general  several   structures  $(V,  \circ,  \Theta)$  of  an
$\mathbb{F}_{p}$-vector space  on the set  $V$ (where $\Theta$  is the
neutral element  for $\circ$), each of  which will yield  in general a
different  copy $\AGL(V, \circ,  \Theta)$ of  the affine  group within
$\Sym(V)$, consisting of the maps $x
\mapsto x g + v$, where $g \in \GL(V, \circ, \Theta)$, and $v \in V$.

In    this   section,    we    assume   that    the    group   $G    =
\Gamma_{\infty}(\mathcal{C})$  generated  by  the round  functions  is
contained in  the affine subgroup $\AGL(V,  \circ, \Theta)$ of
$\Sym(V)$  with  respect  to  the elementary abelian group (i.e.,
$\mathbb{F}_{p}$-vector space) structure  $(V,  \circ,
\Theta)$.  Since  by   our  assumptions  the  group  $T   =  T(V)$  of
translations with respect  to $+$ is contained in  $G$, we obtain that
$T$ is an  abelian regular subgroup of $\AGL(V,  \circ, \Theta)$. This
allows us to use the description of~\cite{CGC-cry-art-cardalsal06} for
this kind of subgroups, which we  revisit in the following. This is an
extension      (and      a      correction)      of      the      work
of~\cite{CGC-cry-art-carantisalaImp}  for  characteristic 2.  Although
the approach  of~\cite{CGC-cry-art-cardalsal06} in terms  of rings has
proved  useful in  other  circumstances, in  this  particular case  we
believe  a  treatment without  rings  to  be  preferable in  terms  of
clarity.

So we have that the translations $\sigma_{y} : x \mapsto x + y$ are in
the affine group $\AGL(V, \circ, \Theta)$.  Thus for $x, y \in
V$, if we  consider the translation $\sigma_{y -  \Theta}$, that takes
$\Theta$ to $y$, we have
\begin{equation}\label{eq:plustocircle}
  x + y - \Theta =x \sigma_{y - \Theta} = x \kappa(y) \circ y,
\end{equation}
for some $\kappa(y) \in  \GL(V, \circ, \Theta)$.  A substitution shows
that
\begin{equation}\label{eq:circletoplus}
  x \circ y = x \kappa(y)^{-1} + y - \Theta.
\end{equation}
In   the   following    we   will   be   using~\eqref{eq:plustocircle}
and~\eqref{eq:circletoplus} repeatedly without further mention.

Note that
\begin{equation*}
  \begin{aligned}
    x + y  + z - 2 \Theta &= (x  + y - \Theta) + (z  - \Theta) \\&= (x
    \kappa(y) \circ y) + z - \Theta \\&= x \kappa(y) \kappa(z) \circ y
    \kappa(z) \circ z
  \end{aligned}
\end{equation*}
and also
\begin{equation}\label{eq:ass2}
  \begin{aligned}
    x + y  + z - 2 \Theta  &= x - \Theta +  (y + z - \Theta)  \\&= x -
    \Theta + (y \kappa(z) \circ  z) \\&= x \kappa(y \kappa(z) \circ z)
    \circ y \kappa(z) \circ z
  \end{aligned}
\end{equation}
so that
\begin{equation}\label{eq:gammaproduct}
  \kappa(y+z  - \Theta)  =  \kappa(y \kappa(z)  \circ  z) =  \kappa(y)
  \kappa(z).
\end{equation}
This shows that $\kappa(V) =\{ \kappa(y)  \mid y \in V\}$ is a group, and
$y \mapsto \kappa(y+\Theta)$ is an epimorphism $(V, +) \to \kappa(V)$,
so that $\kappa(V)$ is a  $p$-group, and thus acts unipotently on $(V,
\circ)$. Note that we have $\kappa(\Theta) = \mathbf{1}$ (set $y = z =
\Theta$ in~\eqref{eq:gammaproduct}), and
\begin{equation}\label{eq:inverse}
  \kappa(y)^{-1} = \kappa(-y + 2 \Theta)
\end{equation}
(set $z = - y + 2 \Theta$ in~\eqref{eq:gammaproduct}).

Now fix $y \in V$, $y \ne \Theta$, such that $y \kappa(x) = y$ for all
$x \in V$. (Since the  group $\kappa(V)$ is unipotent on $(V, \circ)$,
we have $\{y \in  V \mid y \kappa(x) = y \text{ for  all $x \in V$}\} \ne
\{\Theta\}$.)

Note that $\rho$ is $\circ$-affine, so there is a constant $\eta$ such
that $s  \mapsto s \rho  \circ \eta$ is $\circ$-additive.   It follows
that $ (s  \circ t) \rho \circ \eta  = s \rho \circ \eta  \circ t \rho
\circ \eta, $ so that
\begin{equation*}
(s \circ t) \rho = s \rho \circ t \rho \circ \eta
\end{equation*}
for all $s, t \in V$. We use this to compute, for the given $y$ and an
arbitrary $x \in V$, 
\begin{equation}\label{eq:diff}
  \begin{aligned}
    (x + (y - \Theta)) \rho - x \rho &= (y + x - \Theta) \rho - x \rho
    \\&= (y \kappa(x) \circ x) \rho - x \rho \\&= (y \circ x) \rho - x
    \rho \\&=  y \rho \circ x  \rho \circ \eta  - x \rho \\&=  (y \rho
    \circ \eta  ) \circ  x \rho  - x \rho  \\&= (y  \rho \circ  \eta )
    \kappa(x \rho)^{-1} + x \rho -  \Theta - x \rho \\&= (y \rho \circ
    \eta ) \kappa(x \rho)^{-1} - \Theta.
  \end{aligned}
\end{equation}
Write
\begin{equation*}\label{eq:not}
  u = y \rho \circ \eta, \qquad v = - x \rho + 2 \Theta,
\end{equation*}
so       that       $\kappa(v)       =      \kappa(x       \rho)^{-1}$
by~\eqref{eq:inverse}. Now~\eqref{eq:diff} becomes
\begin{equation}\label{eq:morediff}
  (x + (y - \Theta)) \rho - x \rho = u \kappa(v) - \Theta.
\end{equation}

Write $\kappa(v) = \textbf{1} \circ \delta(v)$, that is,
\begin{equation*}
  u \kappa(v) = u \circ u \delta(v)
\end{equation*}
for $u \in V$, where $\delta : V \to \mathrm{End}(V, \circ)$. Then, as
shown in~\cite{CGC-cry-art-cardalsal06}, $\delta$ is $\circ$-additive,
that   is,   $\delta(v_{1}   \circ   v_{2})  =   \delta(v_{1})   \circ
\delta(v_{2})$   for    $v_{1},   v_{2}   \in    V$.   (This   follows
from~\eqref{eq:ass2}, since interchanging $y$  and $z$, and setting $x
= \Theta$, we  obtain $y \kappa(z) \circ z = z  \kappa(y) \circ y$, so
that $y \delta(z) = z \delta(y)$. Now the left-hand side of the latter
equation is  $\circ$-additive in  $y$, and thus  so is  the right-hand
side.)

From the $\circ$-additivity of $\delta$ it follows that
\begin{equation}\label{eq:W}
  W = u \delta(V) = \{ u \delta(v) \mid v \in V \}
\end{equation}
is a $\circ$-subgroup of $V$, and then
\begin{equation*}
  \{ u \kappa(v) \mid v \in V \} =  \{ u \circ u \delta(v) \mid v \in V \} =
  u \circ u \delta(V) = u \circ W
\end{equation*}
is a $\circ$-coset of the $\circ$-subgroup $W$.

We want to show next that  $W$ is invariant under $\kappa(V)$.  For $z
\in V$ we have
\begin{equation}\label{eq:Wisgammainvariant}
  u \delta(v) \kappa(z) = u \delta(v) \circ u \delta(v) \delta(z).
\end{equation}
The first summand  $u \delta(v)$ in the right-hand side  is in $W$; we
want   to  prove   that  also   $   u  \delta(v)   \delta(z)$  is   in
$W$. Now~\eqref{eq:gammaproduct} implies, for $v, z \in V$,
\begin{equation*}
  \delta(v + z  - \Theta) = \delta(v) \circ  \delta(z) \circ \delta(v)
  \delta(z),
\end{equation*}
so that
\begin{equation*}
  \delta(v)  \delta(z)  =  \delta(v  +  z  -  \Theta)  \circ  (\ominus
  \delta(v)) \circ (\ominus \delta(z)),
\end{equation*}
where $\ominus t$  is the opposite of $t$ with  respect to $\circ$. It
follows   that  also   $u  \delta(v)   \delta(z)  \in   W$,   so  that
by~\eqref{eq:Wisgammainvariant}  $u \delta(v)  \kappa(z) \in  W$, that
is, $W$ is $\kappa(V)$-invariant.

From this it follows that $W - \Theta$ is a $+$-subgroup of $V$, as
\begin{equation*}
  (u  \delta(v_{1})  -  \Theta)  +  (u \delta(v_{2})  -  \Theta)  =  u
  \delta(v_{1}) \kappa(u \delta(v_{2})) \circ u \delta(v_{2}) - \Theta
  \in W - \Theta.
\end{equation*}
Now the right-hand side of~\eqref{eq:morediff} reads
\begin{equation}\label{eq:theendisnigh}
  \begin{aligned}
    u  \kappa(v) -  \Theta &=  u  \circ u  \delta(v) -  \Theta \\&=  u
    \delta(v) \circ u - \Theta \\&= u \delta(v) \kappa(u)^{-1} + u - 2
    \Theta.
  \end{aligned}
\end{equation}
So if we take a fixed value of $y$ (and thus of $u$), as chosen above,
and  let   $x$  (and   thus  $v$)  range   in  $V$, we obtain
from~\eqref{eq:morediff} and~\eqref{eq:theendisnigh} that
the set
\begin{equation}\label{coset}
\{(x+(y-\Theta))\rho-x\rho\;:\; x\in V\}
=
  W \kappa(u)^{-1} + u - 2 \Theta = (u - \Theta) + (W - \Theta)
\end{equation}
is a $+$-coset with respect to $u - \Theta$ of the $+$-subgroup $W-
\Theta$ of $V$.

Now  $\lambda$  is additive  with  respect to  $+$,  so,  if we  apply
$\lambda^{-1}$ to~\eqref{coset}, we obtain that
\begin{equation}\label{coset1}
  \{(x+(y-\Theta))\gamma-x\gamma\;:\; x\in V\}
\end{equation}
is  also a  $+$-coset 
of the
$+$-subgroup $W- \Theta$ of $V$.

Now we can choose an index $i$ such that the component $y_i\in V_i$ of
$y  \ne \Theta$  is different  from $\Theta$.  This is  because either
$\Theta =  0$, and then $y \ne  0$ must have a  non-zero component; or
$\Theta \ne  0$, and  then $\Theta$ can  only be  in at most  a single
component  $V_{i_{0}}$,  so in  case  it  suffices  to choose  $i  \ne
i_{0}$. With this choice, we have that the projection
\begin{equation*}
  \{(x+(y_i-\Theta))\gamma_i-x\gamma_i\;:\; x\in V_i\}
\end{equation*}
of the  set (\ref{coset1}) on  $V_i$ is a  $+$-coset of a  subgroup of
$V_i$  with  respect  to $+$  and  so  we  obtain a  contradiction  to
assumption  (2)  of  Theorem~\ref{main}.   

\section{The wreath product case}
\label{sec:wp}
Let $G=\Gamma_{\infty}(\mathcal{C})$ be the wreath product in product action as follows
$$ G=(S_{1}\times\cdots\times S_{c}).O.P,
$$  with   $p^{b}=z^{c}$  for  some  $z$  and   $c>1$.  Here  $T=T_{1}\times
\cdots\times  T_{c}$, with  $T_{i}\leq S_{i}$  and $|T_{i}|=z$  for  each $i$,
$S_{1}\cong\ldots\cong                    S_{c}$,                   $O\leq
\mathrm{Out}(S_{1})\times\cdots\times\mathrm{Out}(S_{c})$,   $P$  permutes
transitively   the   $S_{i}$'s   by   conjugation.   It   follows   that
$S_{1}\times\cdots\times S_{c}=\mathrm{Soc}(G)$.

By Lemma~\ref{lemma:transide}, $G=\langle  T,\rho\rangle$, and $T\leq  \mathrm{Soc}(G)$, so that
$G/\mathrm{Soc}(G)$ is cyclic, spanned  by $\rho$. Moreover, since $P$
permutes transitively  the $S_i$, then $\rho$  permutes cyclically the
$S_{i}$   by  conjugation.   So  we   have,  possibly reordering indices,
$S_{i}^{\rho}=\rho^{-1}S_i\rho=S_{i+1}$   for   each   $i\neq   c$   and
$S_c^{\rho}=S_1$.

Since each  $T_{i}$ is a group of  translations, then $W_{i}=0 T_{i}\subseteq
S_{i}$ is a subgroup  of $V$ of order $z$. But also  $0 S_{i}$ has order $z$,
so  $0 T_{i}=0 S_{i}$ for  each $i$.  Each element  $v$ of  $V$  can be written
uniquely as
$$ 
v=0 t_{1}t_{2}\cdots t_{c}=0 t_{1}+\cdots+0 t_{c}
$$ 
where $t=t_{1}t_{2}\cdots t_{c}$ for unique $t_{i}\in T_{i}$ and so
$$ 
V=W_{1}\oplus\cdots\oplus W_{c}.
$$

For    each    $i$,   $W_{i}\rho=0 S_{i}\rho=0 S_{i+1}^{\rho^{-1}}\rho=0\rho
S_{i+1}=0 S_{i+1}=W_{i+1}$,  since  $0\rho=0$.  Hence  $\rho$  permutes
cyclically the $W_{i}$.  Write $v\in V$ as $v = w_{1} + \cdots + w_{c}$,  with  $w_{i} \in  W_{i}$, and $w_{i} = 0 t_{i}$  for  some  $t_{i} \in T_{i}$. So we have
$$ 
 v\rho=(0 t_{1}+\cdots+ 0 t_{c})\rho=0 t_{1}\cdots t_{c}\rho=0 t_{1}^{\rho}\cdots
t_{c}^{\rho}
$$ 
as the $t_{i}$ are  translations and $0\rho^{-1}=0$. Since $t_{i}^\rho\in
S_i^{\rho}=S_{i+1}$,  there  exist  $t'_{i+1}\in  T_{i+1}$  such  that
$0 t_{i}^\rho=0 t_{i}\rho=0t'_{i+1}\in  W_{i+1}$ (with indices  taken modulo
c), and because $S_{i}$ and $S_{j}$ commute elementwise, we have
\begin{align*}
v   \rho   &   =   0 t_{1}^\rho\cdots  t_{c}^\rho   =   0 t'_{2}t_{2}^\rho\cdots
t_{c}^\rho=0 t_{2}^\rho  t'_{2}\cdots  t_{c}^\rho\\   &  =  0 t'_{3}  t'_{2}  \cdots
t_{c}^\rho  =  0 t'_{2}  t'_{3}\cdots  t_{c}^\rho  =  \ldots\\  &  =0 t'_{1}  t'_{2}
t'_{3}\cdots  t'_{c}= 0 t'_{1}+\cdots+0 t'_{c}\\  &=  0 t_{c}\rho+ 0 t_{1}\rho+\cdots+
0 t_{c-1}\rho\\ &=w_{1}\rho+\cdots+ w_{c}\rho.
\end{align*}
Now we fix an index $\iota$ and we take $u\in W_\iota$. We have
$$ 
v\rho=(w_{1}+\cdots+w_{c})\rho=w_{1}\rho+\cdots+w_{c}\rho
$$ 
where $w_{\iota} \rho \in W_{\iota+1}$. We also have
$$ 
(v+u)\rho=w_{1}\rho+\cdots+(w_\iota+u)\rho+\cdots+w_{c}\rho
$$ 
with $(w_\iota+u)\rho\in W_{\iota+1}$. It follows that
\begin{equation}\label{wi}
(v+u)\rho-v\rho=(w_\iota+u)\rho-w_\iota\rho\in W_{\iota+1}.
\end{equation}
We  recall  that  $\rho=\gamma\lambda$,  with  $\lambda$  additive.  So,
applying $\lambda^{-1}$ to both sides of (\ref{wi}), we obtain
\begin{equation}\label{wi1}
\hat{\gamma}_u(v)=(v+u)\gamma-v\gamma\in W_{\iota+1}\lambda^{-1}.
\end{equation}

Let $\pi_{j}:V\rightarrow V_{j}$ such that $\pi_{j}(v)=v_{j}$ and let $J$ be the set of all $j$ such that $\pi_{j}(W_{\iota})\neq 0$. Now we have two cases. 
\begin{itemize}
\item[(I)]  If $W_{\iota}\cap V_{j}=V_{j}$ for every $j\in J$, then $W_\iota=\bigoplus_{j\in J}V_j$ ,  so $W_\iota\gamma=W_\iota$.  Since  $W_\iota\rho=W_{\iota+1}$, we  have
  $W_\iota=W_\iota\gamma=W_{\iota+1}\lambda^{-1}$     and    so,    by
  (\ref{wi1}), $(v+u)\gamma-v\gamma\in W_\iota$,  for all $v\in V$ and
  $u\in  W_\iota$.  By Proposition~\ref{imp},  it  follows  that $G$  is
  imprimitive, but this contradicts Theorem~\ref{primitivity}.
\item[(II)] Otherwise, there exist  $j$ such that $W_\iota\cap V_j\neq
  V_j$. We denote  $U=W_\iota$ and $U'=W_{\iota+1}\lambda^{-1}$ and we
  note that
\begin{equation}\label{wi2}
(U\cap V_j)\gamma_j=U'\cap V_j.
\end{equation}
We  denote $\gamma_{j}$  with $\gamma'$.  By (\ref{wi1}),  we  have that
$\mathrm{Im}(\hat{\gamma}'_u)\subseteq U'\cap  V_j$ for all  $u\in U\cap
V_j$.  By  assumption  $\gamma'$  is  weakly $p^r$-  uniform,  so,  by
Remark~\ref{rem2},                          $\lvert U'\cap
V_j\rvert=\lvert U\cap  V_j\rvert \geq  p^{m-r}$.  But this contradicts (\ref{wi2}), since $\gamma'$ is strongly $r$-anti-invariant.
\end{itemize}

\section*{Acknowledgements}

The  authors   are  grateful  to  the  referees   for  several  useful
suggestions. We  are particularly indebted  to one of the  referees for
pointing out  an   inconsistency  in  a  previous  version, and  for
suggesting the way to fix it.

\bibliographystyle{amsalpha} 


\bibliography{biblio}

\end{document}